\documentclass[leqno,12pt]{amsart}
\usepackage{amsfonts}
\usepackage{amsmath,amssymb,amsthm}
\usepackage{amsfonts}

\newtheorem{theorem}{Theorem}[section]
\newtheorem{proposition}[theorem]{Proposition}
\newtheorem{lemma}[theorem]{Lemma}

\newtheorem{corollary}[theorem]{Corollary}

\theoremstyle{definition}
\newtheorem{example}[theorem]{Example}
\theoremstyle{definition}
\newtheorem{rem}[theorem]{Remark}
\newtheorem{remark}[theorem]{Remark}
\newtheorem{definition}[theorem]{Definition}

\numberwithin{claim}{theorem}

\renewenvironment{proof}{\textit{Proof.}}{\hfill\ensuremath{\qed}}

\def \qed{\hfill{\hbox{$\square$}}}

\numberwithin{equation}{section}

\begin{document}
\title[Solitons on Hypersurfaces of $\mathbb{Q}^{n}_{\epsilon} \times \mathbb{R}$]{Almost Ricci Solitons on Class $\mathcal A$  
Hypersurfaces of Product Spaces}

\author[A. U. \c Corapl\i]{Ahmet Umut \c Corapl\i}
\address{Department of Mathematics, Faculty of Science and Letters, Istanbul Technical University, \.{I}stanbul, T{\"u}rk\.{I}ye}
\email{corapli18@itu.edu.tr}

\author[B. Bekta\c{s} Dem\.{i}rc\.i]{Burcu Bekta\c s Dem\.{i}rc\.i}
\address{Fatih Sultan Mehmet Vak{\i}f University, Topkap{\i} Campus, Faculty of Engineering,
Department of Software Engineering, \.{I}stanbul, T{\"u}rk\.{I}ye}
\email{bbektas@fsm.edu.tr}

\author[N. Cenk Turgay]{Nurettin Cenk Turgay}
\address{Department of Mathematics, Faculty of Science and Letters, Istanbul Technical University, \.{I}stanbul, T{\"u}rk\.{I}ye}
\email{nturgay@itu.edu.tr}


\begin{abstract}
In this paper, we study hypersurfaces in the product spaces $\mathbb{Q}_{\epsilon}^3 \times \mathbb{R}$ for which the tangential component $T$ of the vector field $\frac{\partial}{\partial t}$ is a principal direction, where $\mathbb{Q}_{\epsilon}^3$ denotes the three-dimensional non-flat Riemannian space form with sectional curvature $\epsilon = \pm 1$, and $\frac{\partial}{\partial t}$ is the unit vector field tangent to the $\mathbb{R}$-factor. We obtain a local classification of hypersurfaces with three distinct principal curvatures satisfying specific functional relations. Then, we determine the necessary and sufficient conditions for such hypersurfaces to admit an almost Ricci soliton structure with potential vector field $T$. Finally, we prove that the only hypersurfaces admitting such solitons are rotational, by showing that the constructed examples with three distinct principal curvatures do not admit almost Ricci solitons.
\end{abstract}

\maketitle

\section{Introduction}
Arising naturally in the study of singularity models and the structure of manifolds evolving under the Ricci flow, \textit{Ricci solitons} are self-similar solutions to the Ricci flow that generalize Einstein metrics and have attracted significant interest due to their applications in geometric analysis and mathematical physics. After Perelman's work \cite{Perelman}, which introduced the entropy formula for the Ricci flow and established its geometric applications, numerous studies have been devoted to the theory and applications of Ricci solitons, \cite{Brozos12,Cao06,Cao09,Chen16}.

From the perspective of submanifold's theory, B.-Y. Chen and S. Deshmukh provided study classification of Ricci solitons 
on Euclidean spaces, specifically the ones naturally arising from the tangential component of position vector field of an hypersurface in  
\cite{Chen_Deshmukh2} and \cite{Chen_Deshmukh1}. Then, H. Al-Sodias \textit{et al.} obtained necessary and sufficient condition for a hypersurface in Euclidean space to be a gradient Ricci soliton in \cite{Sodais}. Moreover, \c{S}.E. Meri\c{c} and E. K{\i}l{\i}\c{c} considered under which condition a submanifold of
a Ricci soliton is also a Ricci soliton and they gave the relation 
between intrinsic and extrinsic invariants of a Riemannian
submanifold which admits a Ricci soliton in \cite{Meric}.
The second author \cite{Demirci} studied Ricci solitons on pseudo-Riemannian hypersurfaces
of 4-dimensional Minkowski spaces.

While the notion of Ricci solitons have a constant soliton parameter, 
this parameter can be relaxed to allow a richer class of geometric structure.
This leads to the notion of an almost Ricci soliton, where the soliton constant is replaced by a smooth function on the manifold.
S. Pigola et al. extended the concept of a gradient Ricci soliton to an almost Ricci soliton in \cite{Pigola}.
Following its introduction, the almost Ricci soliton has become a topic of active research, yielding several important findings, 
 \cite{Barros,Borges22,Guler20}. 

On the other hand, intrinsic and extrinsic properties of hypersurfaces in products of space forms have been studied by many geometers \cite{Dillen09,Dillen12,Fetcu15,Manfio25,MendonTojeiro2014,Tojeiro2010}, where particular attention has been given to hypersurfaces of $\mathbb{Q}^n_\epsilon \times \mathbb{R}$ satisfying certain conditions involving the tangential component $T$ of the vector field defined by the decomposition
\begin{equation}
\label{decompt} 
\frac{\partial}{\partial x_{n+2}} = T + \sigma N,
\end{equation}
where $\mathbb{Q}^n_\epsilon$ denotes either the $n$-dimensional sphere $\mathbb{S}^n$ or the hyperbolic $n$-space $\mathbb{H}^n$, and $\frac{\partial}{\partial x_{n+2}}$ is the vector field tangent to the $\mathbb{R}$-factor of the product space. For example, in \cite{Tojeiro2010}, the following definition is given (See also \cite{MendonTojeiro2014}).
\begin{definition}\cite{Tojeiro2010}
The hypersurface $M$ is said to belong to class $\mathcal A$ if $T$ is a principal direction of $M$.
\end{definition}
Also the local classification of class $\mathcal A$ hypersurfaces were obtained in \cite{Tojeiro2010}. 
Further, in \cite{Dillen09}, Dillen \textit{et al.} studied rotational hypersurfaces in $\mathbb{Q}^n_\epsilon\times\mathbb{R}$ and they proved the following
classification theorem of hypersurfaces with two distinct principal curvatures with a certain functional relation.
\begin{theorem} \label{thmk2eqk3}
\cite{Dillen09}
Take $n \geq 3$ and let $f : M^n \rightarrow \mathbb{Q}^n_\epsilon\times\mathbb{R}$ be a hypersurface with
shape operator
\begin{align}
S=\left(
\begin{array}{ccccc}
\lambda &  &  &  & \\
 &   \mu &  &  & \\
 &  &  \ddots & & \\
 &  &        &   \mu& \\
\end{array}
\right),
\end{align}
with $\lambda\neq\mu$ and suppose that $ST=\lambda T$. Assume moreover that there is a functional relation 
$\lambda(\mu)$. Then, $M$ is an open part of a rotation hypersurface.
\end{theorem}

In this paper, as a continuation of Theorem \ref{thmk2eqk3}, we consider the hypersurfaces in $\mathbb{Q}^3_\epsilon\times\mathbb{R}$ with three distinct principal curvatures $k_1,k_2,k_3$ satisfying functional relations $k_2=f_1(k_1)$ and $k_3 = f_2(k_1)$ for some smooth functions $f_1$ and $f_2$. In particular, we get the local classification of hypersurfaces satisfying this property. 
Then, we investigate whether such hypersurfaces can admit an almost Ricci soliton structure with potential vector field given by  $ T  $. 
Finally, we show that the only hypersurfaces in  $\mathbb Q^3_\epsilon \times \mathbb{R}  $ satisfying  $ ST = \lambda T  $ and admitting an almost Ricci soliton structure are rotational hypersurfaces.


\section{Preliminaries}
Let $\mathbb E^{n+2}_r$ denote the $n+2$-dimensional semi-Euclidean space with the index $r$ given by the metric tensor
$$\widetilde g=\langle \cdot,\cdot\rangle=-\sum\limits_{i=1}^r dx_i^2+\sum\limits_{i=r+1}^n dx_i^2$$
and the Levi-Civita connection $\widehat{\nabla}$. We are going to use the notation $\mathbb E^{n+2}_0=\mathbb{E}^{n+2}$ and 
$\mathbb E^{n+2}_1=\mathbb{L}^{n+2}$ for Euclidean and Minkowski spaces, respectively.

Let $\mathbb{Q}^n_\epsilon$ stand for the $n$-dimensional non-flat Riemannian space with the sectional curvature $\epsilon=\pm1$. 
Throughout this paper, we are going to consider the product spaces $\mathbb{S}^n\times\mathbb{R}$ and $\mathbb{H}^n\times\mathbb{R}$ defined by
\begin{align}
\mathbb{S}^n\times\mathbb{R}&=\{(x_1, \dots, x_{n+2})\in\mathbb{E}^{n+2}\;|\; x_1^2+x_2^2+\cdots+x_{n+1}^2=1\}\\
\mathbb{H}^n\times\mathbb{R}&=\{(x_1, \dots, x_{n+2})\in\mathbb{L}^{n+2}\;|\; -x_1^2+x_2^2+\cdots+x_{n+1}^2=-1,\;x_1>0\}.
\end{align} 
Let $e_{n+2}$ denote the unit normal vector field of $\mathbb{Q}^n_\epsilon$ on the respective flat space $\mathbb{E}^{n+2}$ or $\mathbb{L}^{n+2}$. Note that we have
$$\left.e_{n+2}\right._{(x_1, \dots, x_{n+2})}=(x_1, \dots, x_{n+1}, 0)$$
and $\langle e_{n+2}, e_{n+2}\rangle = \epsilon$.


\subsection{Submanifolds of $\mathbb{Q}^n_\epsilon\times\mathbb{R}$}
Let $M^{n}$ be an oriented hypersurface of $\mathbb{Q}^{n}_{\epsilon} \times \mathbb{R}$ with the unit normal vector field $N$.  The Levi-Civita connections of $M$ and $\mathbb{Q}^n_{\epsilon} \times\mathbb{R}$ are going to be denoted by $\nabla$ and $\widetilde{\nabla}$, respectively. The Gauss and Weingarten formul\ae\ are given by
\begin{align*}
	\tilde{\nabla}_{X}Y=& \nabla_{X}Y + h(X,Y) \\	
 \tilde{\nabla}_{X}N =& -SX, 
\end{align*}
for all vector fields $X$ and $Y$ tangent to $M$, where 
$h$ is the second fundamental form and $S$ is the shape operator of $M$. The second fundamental form and the shape operator are related by 
$$\langle h(X,Y),\eta \rangle  = \langle SX,Y \rangle.$$ 
Furthermore, $\widetilde{\nabla}$ and $\widehat{\nabla}$ have the relation
\begin{align}
\label{LeviCivitaproduct}
\widehat{\nabla}_{Y}Z= \widetilde{\nabla}_{Y}Z - \epsilon (\langle Y,Z \rangle - \langle Y, T \rangle \langle T,Z\rangle)e_{n+2},
\end{align}
where, through a slight misuse of notation, we put $e_{n+2}=\left.e_{n+2}\right|_{M}$.

Let $R$ denote the curvature tensor of $M$ and $\nabla ^\perp h$ stands for the covariant derivative of $h$, that is,
$$(\nabla _{X} ^\perp h) (Y,Z) =\nabla _{X} ^\perp h (Y,Z) - h (\nabla _{X}Y,Z) -  h (Y,\nabla _{X}Z).$$
Then, the Gauss and Codazzi equations 
\begin{align}
\label{Gausseq}
	R(X,Y,Z,W) =& \langle h(Y,Z),h(X,W)\rangle-\langle h(X,Z),h(Y,W)\rangle \\\nonumber&+ \epsilon\langle\left( X\wedge Y + \langle X,T \rangle Y \wedge T - \langle Y,T \rangle X \wedge T\right)Z,W\rangle,  \\
 \label{Codazzieq}
	(\nabla _{X} ^\perp h) (Y,Z) - (\nabla_{Y} ^\perp h) (X,Z) =& \epsilon \langle (X \wedge Y)T,Z \rangle \xi
\end{align}
are  satisfied  for all vector fields $X,Y,Z,W$ tangent to $M$, where we put 
$$(X \wedge Y)Z = \langle Y,Z \rangle X - \langle X,Z \rangle Y.$$

One can define a tangent vector field $T$ and a smooth function $\sigma$ on $M$ by the decomposition \eqref{decompt}. Since $\frac{\partial}{\partial {x_{n+2}}}$ is a parallel vector field in $\mathbb{Q}^{n}_{\epsilon} \times \mathbb{R}$, 
the equations
\begin{align}
\label{paralleltang}	\nabla_{X}T =& \sigma SX,\\
\label{parallelnor}	h(X,T) =& - X(\sigma)N,
\end{align}
are satisfied for all $X \in TM$, \cite{Tojeiro2010}.

\begin{remark}
Throughout this article, by excluding the trivial cases (see \cite{MendonTojeiro2014,Tojeiro2010}), we assume that the vector field $T$ is non-vanishing on $M$. Therefore, we consider an orthonormal frame field $\{e_1, \dots, e_n, N\}$ on $M$ such that $e_1$ is proportional to $T$ with the corresponding  connection 1-forms $\omega_{ij}$ defined by
$$\omega_{ij}(e_k)=\langle \nabla_{e_k} e_i, e_j \rangle.$$
\end{remark}

Therefore, \eqref{decompt} turns into
\begin{equation}
\label{RWTS-etaandTDefNew}
\left.\frac{\partial}{\partial {x_{n+2}}}\right|_M=\cos\theta \, e _1+\sin\theta \, N,  
\end{equation}
for a smooth function $\theta$. From now on, we shall use the indices
$$i,j,k = 1,2,\dots,n \quad \text{and} \quad a,b,c = 2,3,\dots,n.$$
Then, \eqref{paralleltang} gives
\begin{align}
\label{paralleltang1}
	Se_{i} = \cot \theta\nabla_{e_{i}}e_{1}  - e_{i}(\theta)e_{1}.
\end{align}
Note that \eqref{paralleltang1} implies
 \begin{align}
 \label{paralleltang2}
	\langle Se_{i},e_{1} \rangle &=   -e_{i}(\theta), \\
 \label{paralleltang3} 
	\langle Se_{i},e_{a} \rangle &=  \cot \theta \omega_{1a}(e_{i})
\end{align}
from which we have 
 \begin{equation*}
e_{a}(\theta)= - \cot \theta \omega_{1a}(e_{1}),\qquad \omega_{1a}(e_b)=\omega_{1b}(e_a).
\end{equation*}


\subsection{Almost Ricci Soliton}
Let $(M,g)$ be a Riemannian manifold. Then its  Ricci tensor is a symmetric $(0,2)$ tensor defined by 
\begin{equation*}
\mbox{Ric}(X,Y)=\mbox{trace}\{Z\hookrightarrow R(Z,X)Y\}
\end{equation*}
or, equivalently, 
\begin{equation}
\label{Riccicurv}
\mbox{Ric}(X,Y)=\sum_{i=1}^n \langle R(e_i,X)Y, e_i\rangle
\end{equation}
where $e_1, e_2, \dots, e_n$ is an orthonormal frame field of the tangent bundle of $M$. 

A smooth vector field $\xi$ on a Riemannian manifold  defines a \emph{almost Ricci soliton} 
if and only if it satisfies
\begin{align}
\label{defsoliton}
\frac{1}{2} \mathcal{L}_{\xi}g + Ric = \lambda g,
\end{align}
where $\mathcal{L}_{\xi}g$ is the Lie derivative of the metric tensor $g$ with respect to $\xi$, 
$\mbox{Ric}$ is the Ricci tensor of $(M,g)$ and 
$\lambda$ is a smooth function on $M$. An almost Ricci soliton is denoted by $(M,g,\xi,\lambda)$. 

The vector field $\xi$ is called the \emph{potential field} of the almost Ricci soliton.
The Ricci soliton $(M,g,\xi,\lambda)$ is \emph{shrinking, steady} or \emph{expanding} if $\lambda>0$, $\lambda=0$ and $\lambda<0$, respectively.
The almost Ricci soliton $(M,g,\xi,\lambda)$ is called a \emph{gradient Ricci soliton} 
if its potential field $\xi$ is the gradient of a smooth function $f$ on $M$. A gradient Ricci soliton is denoted by $(M,g,f,\lambda)$ and $f$ is called the \emph{potential function}. Note that when $\xi$ is a Killing vector field, i.e., $\mathcal{L}_{\xi} g=0$, 
an almost Ricci soliton $(M, g, \xi, \lambda)$ is an Einstein manifold. For a constant $\lambda$, it becomes a Ricci soliton.


\section{Hypersurfaces of $\mathbb{Q}^{3}_{\epsilon}\times \mathbb{R}$}\label{SectQ3xR-3princip}
In this section, we are going to consider hypersurfaces of 
$\mathbb{Q}^{3}_{\epsilon}\times \mathbb{R}$ such that the tangent vector $T$ defined by \eqref{decompt} 
is an eigenvector of the shape operator $S$. 

\subsection{Examples of Hypersurfaces with Three Distinct Principal Curvatures}
In this subsection, we construct two classes of hypersurfaces of $\mathbb{Q}^3_\epsilon\times\mathbb{R}$ with three distinct principal curvatures.

In the next example,  we have $\varepsilon=1$.
\begin{example}
\label{exS3xR}
Consider the following hypersurface $M$ of $\mathbb{S}^3\times\mathbb{R}$:
\begin{align}\label{positionexs3xR}
\begin{split}{ x}(s,v,w) = &\left(\cos{(\alpha_1(s))}\cos{v}, \cos{(\alpha_1(s))}\sin{v}, 
\sin{(\alpha_1(s))}\cos{w}, \sin{(\alpha_1(s))}\sin{w},\right. \\
&\left.\alpha_{2}(s)\right),
\end{split}
\end{align}
where $\alpha_1$ and $\alpha_2$ are some smooth functions satisfying 
\begin{align}
\label{positionexs3xREq2} \alpha_1'^2(s)+\alpha_2'^2(s)=&1
\end{align}
and $\cos\left(\alpha_1(s)\right), \sin\left(\alpha_1(s)\right) > 0$.

We choose an orthonormal frame field $\{e_1, e_2, e_3, N\}$ on $M$ such that 
$e_1, e_2, e_3$ are tangent to $M$ and $N$ is normal to $M$:
\begin{align}
\begin{split}
\label{orthexS3xR}
e_1=& \frac{\partial}{\partial s},\\
e_2=& \frac{1}{\cos{\alpha_1(s)}}\frac{\partial}{\partial v}, \\
e_3=& \frac{1}{\sin{\alpha_1(s)}}\frac{\partial}{\partial w},\\
N =& \left(\alpha^{\prime}_2 (s)\sin(\alpha_1 (s))\cos v, \alpha^{\prime}_2 (s)\sin(\alpha_1 (s))\sin v,
-\alpha^{\prime}_2 (s)\cos(\alpha_1 (s))\cos w, \right.\\
&\left.-\alpha^{\prime}_2 (s)\cos(\alpha_1 (s))\sin w,\alpha^{\prime}_{1}(s)\right).
\end{split}
\end{align}
The shape operator $S$ of $M$ is given by
\begin{align}
\label{shapeopexS3xR}
S = \begin{pmatrix}
\alpha^{\prime}_1(s)\alpha^{\prime\prime}_2(s)-\alpha^{\prime}_2(s)\alpha^{\prime\prime}_1(s)  & 0 & 0\\
0 & -\alpha^{\prime}_2(s)\tan{(\alpha_1(s))} & 0\\
0 & 0 & \alpha^{\prime}_2(s)\cot{(\alpha_1(s))}
\end{pmatrix}.
\end{align}
Note that the diagonal entries of $S$ are the principal curvatures $k_1$, $k_2$, and $k_3$, respectively. A further computation yields that the vector field $T$ is proportional to $e_1$ and, thus, is a principal direction of $M$. Moreover, we have
\begin{equation*}
\omega_{12}(e_2)= - \alpha^{\prime}_1(s)\tan{(\alpha_1(s))}\mbox{ and }\omega_{13}(e_3)= \alpha^{\prime}_1(s)\cot{(\alpha_1(s))}.
\end{equation*}
\end{example}

We have the next example for the case $\varepsilon=-1$.
\begin{example}
\label{exH3xR}
Consider the following hypersurface $M$ of $\mathbb{H}^3\times\mathbb{R}$:
\begin{align}
\label{positionexH3xR}
\begin{split}
{ x}(s,v,w) =& \left(\cosh(\alpha_1(s))\cosh v, \cosh(\alpha_1(s))\sinh v,\sinh(\alpha_1(s))\cos w,\right.\\&\left.\sinh(\alpha_1(s))\sin w, \alpha_2(s)\right),
\end{split}
\end{align}
where $\alpha_1$ and  $\alpha_2$ are some smooth functions satisfying 
\begin{align}
\label{positionexH3xREq2}
\alpha_1'^2(s)+\alpha_2'^2(s)=1
\end{align}
and $\sinh(\alpha_1(s))>0$.
We choose an orthonormal frame field $\{e_1, e_2, e_3, N\}$ on $M$ such that 
$e_1, e_2, e_3$ are tangent to $M$ and $N$ is normal to $M$:
\begin{align}\label{orthexH3xR}
\begin{split}
    e_1=& \frac{\partial}{\partial s},\\
    e_2=& \frac{1}{\cosh{(\alpha_1(s))}}\frac{\partial}{\partial v}, \\
    e_3=& \frac{1}{\sinh{(\alpha_1(s)})}\frac{\partial}{\partial w},\\
    N =& \left(-\alpha'_2(s)\sinh{(\alpha_1(s))}\cosh{v} , -\alpha'_2(s)\sinh{(\alpha_1(s))}\sinh{v} ,
    -\alpha'_2(s)\cosh{(\alpha_1(s))}\cos{w} ,\right.\\
		&\left.-\alpha'_2(s)\cosh{(\alpha_1(s))}\sin{w}, \alpha'_1(s)\right).
 \end{split}   
\end{align}
The shape operator $S$ of $M$ is given by
\begin{align}
\label{shapeopexH3xR}
S = \begin{pmatrix}
\alpha'_1(s)\alpha''_2(s)-\alpha'_2(s)\alpha''_1(s)  & 0 & 0\\
0 & \alpha'_2(s)\tanh{(\alpha_1(s))} & 0\\
0 & 0 & \alpha'_2(s)\coth{(\alpha_1(s))}
\end{pmatrix}
\end{align}
Note that the diagonal entries of $S$ are the principal curvatures $k_1$, $k_2$, and $k_3$, respectively. A further computation yields that the vector field $T$ is proportional to $e_1$ and, thus, is a principal direction of $M$. Moreover, we have
\begin{equation*}
\omega_{12}(e_2)= \alpha^{\prime}_1(s)\tanh\left(\alpha_1(s)\right)\mbox{ and }\omega_{13}(e_3)= \alpha^{\prime}_1(s)\coth\left(\alpha_1(s)\right).
\end{equation*}
\end{example}


\subsection{Local Classification Theorem}
In this subsection, we consider the hypersurfaces in $\mathbb{Q}^3_\epsilon\times\mathbb{R}$ with three distinct principal curvatures which satisfy certain functional relations.

Let $M$ be a hypersurface of 
$\mathbb{Q}^{3}_{\epsilon}\times \mathbb{R}$ whose shape operator  has the matrix representation
\begin{equation}\label{shapeop_m}
S=\left(
\begin{array}{ccc}
k_1 & 0 & 0\\
0 & k_2 & 0 \\
0 & 0 & k_3
\end{array}
\right),
\end{equation}
with respect to $\{e_1, e_2, e_3\}$.

\begin{remark}\label{RemarkSct32Rem1}
In the remaining part of this section, we are going to assume that $e_1$ is the tangent vector field defined by \eqref{RWTS-etaandTDefNew}, $k_i-k_j$ does not vanish on $M$ whenever $i\neq j$  and there exist some smooth functions $f_1$ and $f_2$ such that $k_2=f_1(k_1)$ and $k_3=f_2(k_1)$. 
\end{remark}

In this case, from \eqref{paralleltang2} and \eqref{paralleltang3} we have
\begin{align}
\label{Sect3AraDenk1}
\begin{split}
k_1=-e_1(\theta), \quad  k_a=\cot \theta \omega_{1a}(e_a), \quad e_a(\theta)=0, \quad \omega_{1a}(e_1)=\omega_{1a}(e_b)=0\mbox{ if $a\neq b$},
\end{split}
\end{align}
respectively.  Furthermore, by combining the Codazzi equation \eqref{Codazzieq} with \eqref{shapeop_m} and \eqref{Sect3AraDenk1}, we obtain
\begin{subequations}\label{CodazziclassA} to get
\begin{eqnarray}
\label{CodazziclassA1}	\omega_{23}(e_1)&=&0,\\
\label{CodazziclassA2}  e_2(k_1)=e_3(k_1)&=&0,\\
\label{CodazziclassA3}	e_1(k_2)+(k_2-k_1)\omega_{12}(e_2)&=&-\epsilon\cos{\theta}\sin{\theta},\\
\label{CodazziclassA4} e_3(k_2)&=&(k_2-k_3)\omega_{23}(e_2),\\
\label{CodazziclassA5} e_1(k_3)+(k_3-k_1)\omega_{13}(e_3)&=&-\epsilon\cos{\theta}\sin{\theta},\\
\label{CodazziclassA6} e_2(k_3)&=&(k_2-k_3)\omega_{23}(e_3).
\end{eqnarray}
\end{subequations}
Note that the functional relations between the principal curvatures given in Remark \ref{RemarkSct32Rem1}  imply $e_a(k_b) = 0$ because of \eqref{CodazziclassA2}. Therefore, from \eqref{CodazziclassA2} and  \eqref{CodazziclassA4} we have
\begin{equation}\label{Sect3AraDenk2}
\omega_{23}(e_2)=\omega_{23}(e_3)=0.
\end{equation}

By summing up \eqref{Sect3AraDenk1}, \eqref{CodazziclassA1} and \eqref{Sect3AraDenk2} with the Gauss Formula and \eqref{LeviCivitaproduct}, we obtain
\begin{subequations}\label{ConhatclassAeqNALL}
\begin{eqnarray}
\label{ConhatclassAeqN1} \widehat{\nabla}_{e_1}e_1 &=& -e_1(\theta) {N}- \epsilon \sin^2 \theta e_5,\\
\label{ConhatclassAeqN2} \widehat{\nabla}_{e_1}e_2 &=& \widehat{\nabla}_{e_1}e_3 =0, \\
\label{ConhatclassAeqN2b} \widehat{\nabla}_{e_2}e_3 &=& \widehat{\nabla}_{e_3}e_2 =0, \\
\label{ConhatclassAeqN3} \widehat{\nabla}_{e_2}e_1 &=&  \omega_{12}(e_2)e_2, \\
\label{ConhatclassAeqN4} \widehat{\nabla}_{e_2}e_2 &=& -\omega_{12}(e_2)e_1 + \cot \theta \omega_{12}(e_2)N - \epsilon e_5,\\
\label{ConhatclassAeqN5} \widehat{\nabla}_{e_3}e_1 &=& \omega_{13}(e_3)e_3,\\
\label{ConhatclassAeqN6} \widehat{\nabla}_{e_3}e_3 &=& -\omega_{13}(e_3)e_1 + \cot \theta \omega_{13}(e_3)N - \epsilon e_5.
\end{eqnarray}
\end{subequations}

In the next lemma, we construct a local coordinate system on $M$.
\begin{lemma}\label{Lemma1ofSect32}
Let $M$ be a hypersurface of $\mathbb{Q}^3_\epsilon\times\mathbb{R}$ satisfying the assumptions given in Remark \ref{RemarkSct32Rem1} and $p\in M$. Then, there exists a local coordinate system $(\mathcal O_p,(s,v,w))$ such that $p\in\mathcal O_p$ and 
\begin{align}\label{Lemma1ofSect32Eq1}
\begin{split}
e_1&=\frac{\partial}{\partial s},\\
e_2&=\frac{1}{\phi_2(s)}\frac{\partial}{\partial v},\\
e_3&=\frac{1}{\phi_3(s)}\frac{\partial}{\partial w},
\end{split}
\end{align}
where $\phi_2,\phi_3$ are some smooth functions satisfying
\begin{eqnarray}
\label{Lemma1ofSect32Eq3} \left.\omega_{12}(e_2)\right|_{\mathcal O_p}=\frac{\phi_2'}{\phi_2},&\qquad& 
\left.\omega_{13}(e_3)\right|_{\mathcal O_p}=\frac{\phi_3'}{\phi_3}.
\end{eqnarray}
\end{lemma}

\begin{proof}
By using \eqref{ConhatclassAeqN2}, \eqref{ConhatclassAeqN3} and \eqref{ConhatclassAeqN5}, we obtain
\begin{align}\label{Lemma1ofSect32Eq1AraDenk0}
\begin{split}
    &[e_1,e_2] =  - \omega_{12}(e_2)e_2, \\
    &[e_1,e_3] =  -\omega_{13}(e_3)e_3, \\
    &[e_2,e_3] =  0.
    \end{split}
\end{align}

On the other hand, because of \eqref{Sect3AraDenk1}, \eqref{Sect3AraDenk2} and Remark \ref{RemarkSct32Rem1}, we have 
\begin{align*}
   &e_a(\omega_{12}(e_2)) = e_a(k_2 \tan \theta) = \tan \theta e_a(k_2) = \tan \theta e_a(f_1(k_1)) = 0, \\
   &e_a(\omega_{13}(e_3)) = e_a(k_3 \tan \theta) = \tan \theta e_a(k_3) = \tan \theta e_a(f_2(k_1)) = 0.
\end{align*}
Let $\phi_2$ and $\phi_3$ be functions defined by
\begin{align}
\label{phi_defn}
    \phi_2 = e^{\zeta_2}, \quad    \phi_3 = e^{\zeta_3},
\end{align}
where $\zeta_2$ and $\zeta_3$ are smooth functions satisfying $e_1(\zeta_a) = \omega_{1a}(e_a)$ and $e_2(\zeta_a) =e_3(\zeta_a)=0$. From \eqref{phi_defn}, $\phi_2$ and $\phi_3$ satisfy
\begin{equation}\label{Lemma1ofSect32Eq1AraDenk1}
e_1(\phi_a)-\phi_a\omega_{1a}(e_a)=e_2(\phi_a)=e_3(\phi_a)=0.
\end{equation}
By combining \eqref{Lemma1ofSect32Eq1AraDenk0} and \eqref{Lemma1ofSect32Eq1AraDenk1} we obtain 
$$[X,Y] = [X, Z]= [Y,Z]=0,$$
where $X,Y,Z$ are vector fields defined by $X= e_1, \; Y=\phi_2 e_2, \; Z=\phi_3 e_3$. Consequently, on a neighborhood $\mathcal O_p$ of $p$, there exists a local coordinate system $(s,v,w)$ such that $X=\frac{\partial}{\partial s}$, $Y=\frac{\partial}{\partial v}$ and $Z = \frac{\partial}{\partial w}$. \eqref{Sect3AraDenk1} and \eqref{Lemma1ofSect32Eq1AraDenk1}  imply $\theta=\theta(s)$ and $\phi_a=\phi_a(s)$ on $\mathcal O_p$, respectively. Hence,  we have \eqref{Lemma1ofSect32Eq1}. On the other hand, \eqref{Lemma1ofSect32Eq1AraDenk1} implies \eqref{Lemma1ofSect32Eq3}.
\end{proof}

Next, by considering Lemma \ref{Lemma1ofSect32}, we obtain a local parametrization of $\mathcal O_p$. 
\begin{lemma}\label{Lemma2ofSect32}
Let $M$ be a hypersurface of $\mathbb{Q}^3_\epsilon\times\mathbb{R}$ satisfying the assumptions given in Remark \ref{RemarkSct32Rem1}, $p\in M$ and $(\mathcal O_p,(s,v,w))$ be local frame field constructed in Lemma \ref{Lemma1ofSect32}. Then, $\mathcal O_p$ can be parametrized as
\begin{align}
\label{positionvec1}
	{ x}(s,v,w)  = { y}(s,v) + { z}(s,w),
\end{align}
where ${ y}(s,v)$ and ${ z}(s,w)$ are smooth vector-valued functions that satisfy
\begin{align}
\label{connewe1_1}
    &{ y}_{v}(s,v) = \phi_2(s)\gamma(v)\\
\label{conneweq2_2}
    &{ z}_{w}(s,w) = \phi_3(s)\beta(w)
\end{align}
for some smooth functions $\gamma(v)$, $\beta(w)$ such that
\begin{align}
\label{eqgamma}
\gamma''(v) + \left(\frac{1}{\sin^2 \theta}(\phi_{2}')^2 + \epsilon (\phi_2)^2\right) \gamma(v)&=0,\\ 
\label{eqbeta}
 \beta''(w) + \left(\frac{1}{\sin^2 \theta}(\phi_{3}')^2 + \epsilon (\phi_3)^2\right) \beta(w)&=0.
\end{align}
\end{lemma}
\begin{proof}
Let $x$ be the position vector of $M$ on the flat ambient space. Then, by a straightforward computation using \eqref{ConhatclassAeqN2b}, we obtain
$x_{wv} = 0$ which yields \eqref{positionvec1} for some smooth vector-valued functions $y = y(s,v)$ and $z = z(s,w)$.
Furthermore, by combining \eqref{ConhatclassAeqN2}, \eqref{ConhatclassAeqN4} and \eqref{ConhatclassAeqN6} with \eqref{positionvec1}, we obtain
\begin{eqnarray}
\label{connewe1}{ y}_{vs} &=& \frac{\phi_2'}{\phi_2} { y}_v,\\
\label{conneweq2}{ z}_{ws} &=& \frac{\phi_3'}{\phi_3} { z}_w,\\
\label{conneweq3}{y}_{vv} &=& {\phi_2^2 } (-\omega_{12}(e_2)e_1 + \cot \theta \omega_{12}(e_2)N - \epsilon e_5),\\
\label{conneweq4}{ z}_{ww} &=& {\phi_3^2}(-\omega_{13}(e_3)e_1 + \cot \theta \omega_{13}(e_3)N - \epsilon e_5).
\end{eqnarray}
Note that \eqref{connewe1} and \eqref{conneweq2} imply \eqref{connewe1_1} and \eqref{conneweq2_2}, respectively, for some smooth functions $\gamma$ and $\beta$.

On the other hand, by taking the covariant derivative of both sides of \eqref{conneweq3} along $e_2$, and of \eqref{conneweq4} along $e_3$, and by using \eqref{ConhatclassAeqNALL}, we obtain
\begin{eqnarray}
\label{eqy}
y_{vvv} + \left(\frac{1}{\sin^2 \theta}(\phi'_{2})^2 + \epsilon (\phi_2)^2\right)y_v &=& 0, \\
\label{eqz}
z_{www} + \left(\frac{1}{\sin^2 \theta}(\phi'_{3})^2 + \epsilon (\phi_3)^2\right)z_w &=& 0,
\end{eqnarray}
respectively.

Finally, by combining \eqref{connewe1_1} with \eqref{eqy}, we obtain \eqref{eqgamma}, whereas \eqref{eqz}, together with \eqref{conneweq2_2}, implies \eqref{eqbeta}.
\end{proof}

\begin{rem}\label{C_D_const} 
\eqref{eqy} and \eqref{eqz} imply that there exists some constants $c,d\in\mathbb R$ such that
   the functions 
\begin{eqnarray}
\label{C_Constant}    \left(\frac{1}{\sin^2 \theta}(\phi'_{2})^2 + \epsilon (\phi_2)^2\right) &=& c, \\
\label{D_Constant}    \left(\frac{1}{\sin^2 \theta}(\phi'_{3})^2 + \epsilon (\phi_3)^2\right) &=& d.
\end{eqnarray}
\end{rem}

\textbf{The case $\epsilon=1$. }Now, assume that $M$ is a hypersurface of $\mathbb{S}^3\times \mathbb{R}$. Then, the constants $c$ and $d$ appearing in \eqref{C_Constant} and \eqref{D_Constant} are positive. So, we put $c=A^2$ and $d=B^2$. Then, by solving \eqref{eqgamma} and \eqref{eqbeta}, we obtain
$$\gamma(v) = \cos(Av) C_1+ \sin(Av)C_2$$ and $$\beta(w) =\cos(Bw)  D_1+ \sin(Bw)D_2,$$ respectively,  
where $C_1$, $C_2$, $D_1$, $D_2$ are constant vectors in  $\mathbb{E}^5$. Consequently, \eqref{connewe1_1} and \eqref{conneweq2_2} give
\begin{align}\label{eqy1}
\begin{split}
 y(s,v) =& \frac{\phi_2(s)}{A}\sin{(Av)}C_1 - \frac{\phi_2(s)}{A}\cos{(Av)}C_2 + C_3(s), \\
 z(s,w) =& \frac{\phi_3(s)}{B}\sin{(Bw)}D_1 - \frac{\phi_3(s)}{B}\cos{(Bw)} D_2+ D_3(s), 
\end{split}
\end{align}
where $C_3(s)$ and $D_3(s)$ are some smooth $\mathbb{E}^5$-valued functions. 

Since $\{e_1, e_2, e_3\}$ is an orthonormal tangent frame field of $M$, by a direct computation using \eqref{Lemma1ofSect32Eq1} and \eqref{eqy1}, we observe that $\{C_1,C_2,D_1,D_2,\frac{\partial}{\partial x_5}\}$ is an orthonormal basis for $\mathbb{E}^5$. 
Therefore, up to a linear isometry of $\mathbb{S}^3\times \mathbb{R}$, we assume that
\begin{equation}\label{c1c2c3c4epsilon1}
C_1=\frac{\partial}{\partial x_2},\ C_2=-\frac{\partial}{\partial x_1},\ D_1=\frac{\partial}{\partial x_4},\ D_2=-\frac{\partial}{\partial x_3}
\end{equation}
In this case, by combining \eqref{positionvec1} with \eqref{eqy1} and \eqref{c1c2c3c4epsilon1}, we get 
\begin{align}\label{positionvec2S31}
    { x}(s,v,w) &=\left(\frac{\phi_2(s)}{A}\cos{(Av)},\frac{\phi_2(s)}{A}\sin{(Av)},\frac{\phi_3(s)}{B}\cos{(Bw)},
    \frac{\phi_3(s)}{B}\sin{(Bw)},0\right) \\ \nonumber
    &+\Gamma(s),
\end {align}
where we put $\Gamma = C_3 + D_3=\left(\Gamma_1,\Gamma_2,\Gamma_3,\Gamma_4,\Gamma_5\right)$. Moreover, since $M$ lays on $\mathbb{S}^3\times\mathbb{R}$, we have $\Gamma_1 = \Gamma_2 = \Gamma_3 = \Gamma_4 = 0$ and 
\begin{align}\label{Be4classS3xRthreedistinctEq1}
    \frac{\phi_2^2(s)}{A^2}+\frac{\phi_3^2(s)}{B^2}=1
\end{align}
and $\langle e_1,e_1 \rangle = 1$ implies
\begin{align}\label{Be4classS3xRthreedistinctEq2}
\frac{\left(\phi'{}_2(s)\right)^2}{A^2} + \frac{\left(\phi'{}_3(s)\right)^2}{B^2} + \left(\Gamma^{\prime}_5(s)\right)^2  = 1. 
\end{align}
Consequently, \eqref{positionvec2S31} turns into 
\begin{align*}
    { x}(s,v,w) = \left(\frac{\phi_2(s)}{A}\cos{(Av)},\frac{\phi_2(s)}{A}\sin{(Av)},\frac{\phi_3(s)}{B}\cos{(Bw)},\frac{\phi_3(s)}{B}\sin{(Bw)},\Gamma_5(s)\right).
\end{align*}
Next, by considering \eqref{Be4classS3xRthreedistinctEq1} we define a function $\alpha_1$ by
$$\frac{\phi_2(s)}{A}=\cos(\alpha_1(s))\mbox{ and }\frac{\phi{}_3(s)}{B}=\sin(\alpha_1(s))$$
and put $\Gamma_5=\alpha_2$. After a suitable scaling on the parameters $v$ and $w$, 
we observe that $\mathcal O_p$ can be parametrized as given in \eqref{positionexs3xR}. Furthermore, \eqref{Be4classS3xRthreedistinctEq2} turns into \eqref{positionexs3xREq2}. Hence, we have the following local classification theorem for hypersurfaces of $\mathbb{S}^3\times \mathbb{R}$.
\begin{theorem}
\label{classS3xRthreedistinct}
Let ${ x}:M\rightarrow \mathbb{S}^3\times \mathbb{R}\subset\mathbb{E}^5$ be a hypersurface
with three distinct principle curvatures $k_1$, $k_2$ and $k_3$ 
satisfying $ST=k_1 T$.
Assume that there are functional relations $k_2=f_1(k_1)$ and $k_3 = f_2(k_1)$ for some functions $f_1$ and $f_2$.
Then, $M$ is locally congruent to the hypersurface given in Example \ref{exS3xR}. 
\end{theorem}


\textbf{The case $\epsilon=-1$. }Now, assume that $M$ is a hypersurface of $\mathbb{H}^3\times \mathbb{R}$. First, we get the following lemma:
\begin{lemma}\label{Lemmaeps-1constscd}
    Suppose that $M$ is a hypersurface of $\mathbb{H}^3\times \mathbb{R}$, that is, $\epsilon = -1$. Then, the constants $c$ and $d$ appearing in \eqref{C_Constant} and \eqref{D_Constant} have opposite signatures.
\end{lemma}
\begin{proof}
By using \eqref{ConhatclassAeqN2b}, \eqref{ConhatclassAeqN4}, \eqref{ConhatclassAeqN5} and \eqref{Lemma1ofSect32Eq1AraDenk0}, we observe that the Gauss equation \eqref{Gausseq} with $ X = e_2 $, $ Y = e_3 $, and $ Z = e_2 $, gives
$$\omega_{12}(e_2)\omega_{13}(e_3) + k_2k_3 = 1,$$  
which, together with \eqref{Sect3AraDenk1}, yields  
\begin{equation} \label{classH3xRthreedistinctDenk1}
k_2k_3 = \cos^2\theta.
\end{equation}
On the other hand, \eqref{Sect3AraDenk1}, \eqref{Lemma1ofSect32Eq3}, \eqref{C_Constant} and \eqref{D_Constant} imply
\begin{equation}\label{classH3xRthreedistinctDenk2}
c=(\phi_2(s))^2 (k_2^2\sec^2\theta -1) ,\qquad d=(\phi_3(s))^2 (k_3^2\sec^2\theta -1)
\end{equation}
By a direct computation using \eqref{classH3xRthreedistinctDenk1} and \eqref{classH3xRthreedistinctDenk2}, we get
\begin{equation}\label{classH3xRthreedistinctDenk3}
-\frac{\phi_3^2\cos^2\theta}{\phi_2^2k_2^2}c=d.
\end{equation}
Note that if $c=d=0$ then, \eqref{classH3xRthreedistinctDenk1} and \eqref{classH3xRthreedistinctDenk2} imply
$$k_2=k_3=\cos\theta$$
which is a contradiction because of the assumption $k_2 \neq k_3$. Therefore, \eqref{classH3xRthreedistinctDenk3} yields that $c$ and $d$ have opposite signatures. 
\end{proof}

As a consequence of Lemma \ref{Lemmaeps-1constscd}, without loss of generality, we put $c=-A^2$ and $d=B^2$ for some positive $A,B\in\mathbb R$. Therefore, similar to the case $\epsilon=1$, from \eqref{eqgamma} and \eqref{eqbeta} we obtain that 
\begin{align}\label{eqy2}
\begin{split}
 y(s,v) =& \frac{\phi_2(s)}{A}\sinh{(Av)}C_1 + \frac{\phi_2(s)}{A}\cosh{(Av)}C_2 + C_3(s), \\
 z(s,w) =& \frac{\phi_3(s)}{B}\sin{(Bw)}D_1 - \frac{\phi_3(s)}{B}\cos{(Bw)} D_2+ D_3(s), 
\end{split}
\end{align}
 for some constant vectors $C_1,C_2,D_1,D_2\in\mathbb{E}^5_1$, where $C_3(s)$ and $D_3(s)$ are some smooth $\mathbb{E}^5_1$-valued functions.

Moreover, since $\{e_1, e_2, e_3\}$ is an orthonormal tangent frame field of $M$, by a direct computation using \eqref{Lemma1ofSect32Eq1} and \eqref{eqy2}, we observe that $\{C_2,C_1,D_1,D_2,\frac{\partial}{\partial t}\}$ is an orthonormal basis for $\mathbb{E}^5_1$ such that $C_2$ is time-like. So, up to linear isometry, one can assume $C_1, C_2, D_1$ and $D_2$ as
\begin{equation*}
C_2=\frac{\partial}{\partial x_1},\ C_1=\frac{\partial}{\partial x_2},\ D_1=\frac{\partial}{\partial x_4},\ D_2=-\frac{\partial}{\partial x_3}.
\end{equation*}
Hence, the position vector field ${ x}$ of $M$ can be expressed as 
\begin{align}
    { x}(s,v,w) &=\left(\frac{\phi_2(s)}{A}\cosh{(Av)},\frac{\phi_2(s)}{A}\sinh{(Av)},
    \frac{\phi_3(s)}{B} \cos{(Bw)}, \frac{\phi_3(s)}{B} \sin{(Bw)},0 \right) \\ \nonumber
    &+\Gamma(s),
\end {align}
where we put $\Gamma = C_3 + D_3=\left(\Gamma_1,\Gamma_2,\Gamma_3,\Gamma_4,\Gamma_5\right)$.

By a direct computation similar to the case $\epsilon=1$, we observe that $\mathcal O_p$ can be parametrized as given in \eqref{positionexH3xR} for some smooth functions $\alpha_1,\alpha_2$ satisfying \eqref{positionexH3xREq2}. Hence, we get the following theorem.
\begin{theorem}
\label{classH3xRthreedistinct}
Let $x:M\rightarrow \mathbb{H}^3\times \mathbb{R}$ be a class $\mathcal{A}$ hypersurface with three distinct principle curvatures $k_1$, $k_2$ and $k_3$ satisfying $ST=k_1 T$. Assume that there are functional relations with functional relations $k_2=f_1(k_1)$ and $k_3 = f_2(k_1)$ for some functions $f_1$ and $f_2$. Then, it is an open part of the hypersurface given in Example \ref{exH3xR}. 
\end{theorem}

\section{Almost Ricci Solitons on Hypersurfaces of $\mathbb{Q}^{n}_\epsilon \times \mathbb{R}$}
In this section, we consider almost Ricci soliton on hypersurfaces of $\mathbb{Q}^{3}_\epsilon \times \mathbb{R}$
satisfying $ST=k_1 T$, where the potential vector field is taken to be the tangential part $T$ of $\frac{\partial}{\partial t}$. 

\begin{rem}
The vector field $T$ in \eqref{decompt} is the gradient of the height function 
\begin{equation}\label{Sect4Rem1Eq1}
h = \left \langle  f, \frac{\partial}{\partial x_{n+2}} \right \rangle,
\end{equation}
where $f \colon M^n \to \mathbb{Q}^n_{\epsilon} \times \mathbb{R}$ is the inclusion of hypersurface, \cite{MendonTojeiro2014}. Therefore, an almost Ricci soliton
with the potential vector field $T$ is a gradient soliton $(M,g,h,\lambda)$.
\end{rem}
First, we compute the Lie derivative of the metric $g$ and 
the components of Ricci tensor of any hypersurfaces in $\mathbb{Q}^{n}_\epsilon \times \mathbb{R}$.  
\begin{lemma}
\label{lemmaLiederiv}
Let $M$ be a hypersurface of  $\mathbb{Q}^{n}_\epsilon \times \mathbb{R}$. Then, the Lie derivative of the metric $g$ along $T$ satisfies 
the followings:
\begin{align}
\label{Liederv1}
    (\mathcal{L}_T g)(e_1, e_1)= -2\sin{\theta} e_1(\theta),\\
\label{Liederv2}
    (\mathcal{L}_T g)(e_i, e_a)= 2\cos{\theta} \omega_{1a}(e_i),
\end{align}
where $\theta$ is defined by \eqref{RWTS-etaandTDefNew} for $i=1,2,\dots,n$ and $a=2,\dots,n$.  
\end{lemma}

\begin{proof}
Suppose that $M$ is a hypersurface of  $\mathbb{Q}^{n}_\epsilon \times \mathbb{R}$. 
Then, from the definition of Lie derivative, we have
\begin{align}
\label{Liedervcal1}
	 (\mathcal{L}_{T}g)(X,Y) = \langle \nabla_{X}T,Y \rangle + \langle X,\nabla_{Y}T \rangle
\end{align}
for tangent vectors $X,Y$. 
Considering the equation \eqref{paralleltang}, \eqref{Liedervcal1} becomes 
\begin{align}
\label{Liedervcal2}
	(\mathcal{L}_{T}g)(X,Y) = 2 \sin \theta \langle SX,Y \rangle. 
\end{align}
From \eqref{paralleltang2} and \eqref{paralleltang3} in \eqref{Liedervcal2} for the orthonormal frame field $\{e_1, e_2, \dots, N\}$, we get desired equations. 
\end{proof}

\begin{lemma}
\label{lemmaRiccitensor}
Let $M$ be a hypersurface of $\mathbb{Q}^{n}_\epsilon \times \mathbb{R}$. 
Then, the components of the Ricci tensor of $M$ satisfies the following
equations:
\begin{align}
\label{Riccicurv1}
	Ric(e_1,e_1) &= -e_{1}(\theta) \cot \theta \omega_{1a}(e_a) - \cot^{2} \theta \omega_{1a}(e_1) ^2  + \epsilon  (n-1) \sin^2 \theta, \\
 \label{Riccicurv2}
	Ric(e_1,e_a) &= \cot^{2}\theta \left(\omega_{1a}(e_1)\omega_{1b}(e_b) - \omega_{1a}(e_b)\omega_{1b}(e_1)\right), \\
\label{Riccicurv3}
        Ric(e_a,e_b) &= -\cot^{2}\theta \omega_{1a}(e_1) \omega_{1b}(e_1)  -\cot \theta \omega_{1a}(e_b)e_{1}(\theta)  \\
	&+ \cot^{2}\theta \left(\omega_{1a}(e_b)\omega_{1c}(e_c) - \omega_{1a}(e_c)\omega_{1b}(e_c)\right) - \epsilon\delta_{ab}(n-2-\sin^2{\theta}).  \nonumber
\end{align}
where $\{e_1, e_2,\dots, e_n\}$ is a local orthonormal frame on $M$ and $a,b,c=2,\dots,n$. 
\end{lemma}

\begin{proof}
By a direct calculation, we have the following nontrivial possibilities for the Riemannian curvature tensor in \eqref{Gausseq}:
\begin{align}
\label{curv1}
    R(e_1,e_a)e_1 &= \cot \theta \left(e_{1}(\theta) \omega_{1a}(e_b) +  \cot \theta \omega_{1a}(e_{1}) \omega_{1b}(e_1) \right) e_b
     - \epsilon \sin^{2} \theta e_{a},\\	
\label{curv2} 
    R(e_1,e_a)e_b  &= (\epsilon\sin^2{\theta}\delta_{ab}-\cot^2{\theta} \omega_{1a}(e_1)\omega_{1b}(e_1)-\cot{\theta} e_1(\theta) \omega_{1a}(e_b)) e_1 \notag\\          &+ \cot^2{\theta}(\omega_{1a}(e_b)\omega_{1c}(e_1)-\omega_{1b}(e_1)\omega_{1a}(e_c))e_c,\\
 \label{curv4}
    R(e_a,e_b)e_c &= \cot^{2}\theta (\omega_{1b}(e_{c})\omega_{1a}(e_{1})- \omega_{1a}(e_c) \omega_{1b}(e_1))e_1 \notag  \\
	&+ \cot^{2}\theta (\omega_{1b}(e_{c}) \omega_{1a}(e_d) - \omega_{1a}(e_{c}) \omega_{1b}(e_d))e_d + 
        \epsilon (\delta_{bc}e_a-\delta_{ac}e_b ).
\end{align}
From \eqref{Riccicurv}, we deduce the equations in Lemma \ref{lemmaRiccitensor}.  
\end{proof}

Combining the equations in Lemma \ref{lemmaLiederiv} and Lemma \ref{lemmaRiccitensor} with \eqref{defsoliton}, we get the following theorem. 

\begin{theorem}
\label{thmsolcondndim}
Let $M$ be a hypersurface of $\mathbb{Q}^n_\epsilon\times\mathbb{R}$. Then, $(M,g)$ admits an almost Ricci soliton $(M,g,T,\lambda)$ 
if and only if the following equations are satisfied:
\begin{align}
\label{solitoneq1}
	&-\sin \theta e_{1}(\theta) - e_{1}(\theta) \cot \theta  \omega_{1a}(e_a) 
 - \cot^{2} \theta  (\omega_{1a}(e_1))^2 + (n-1) \epsilon \sin^2 \theta = \lambda, \\
\label{solitoneq2}
	&\cos \theta \omega_{1a}(e_1) + \cot^{2}\theta \left(\omega_{1a}(e_1)\omega_{1b}(e_b) - \omega_{1a}(e_b)\omega_{1b}(e_1)\right) = 0, \\
\label{solitoneq3}
	&\cos \theta \omega_{1a}(e_b) -\cot^{2}\theta \omega_{1a}(e_1) \omega_{1b}(e_1)  -\cot \theta \omega_{1a}(e_b)e_{1}(\theta) \notag \\
	& + \cot^{2}\theta \left(\omega_{1a}(e_b)\omega_{1c}(e_c) - \omega_{1a}(e_c)\omega_{1b}(e_c)\right)-\epsilon\delta_{ab} (n-2-\sin^{2}\theta)= \delta_{ab}\lambda 
\end{align}
for $a,b,c=2,\dots,n$.
\end{theorem}

From now on, we study almost Ricci soliton on a hypersurface of $\mathbb{Q}^3_\epsilon\times\mathbb{R}$
satisfying $ST=k_1 T$. From Section \ref{SectQ3xR-3princip}, we have the principal curvatures and the connection forms of such hypersurfaces. 
Thus, Theorem \ref{thmsolcondndim} implies the following corollary. 

\begin{corollary}
\label{almostsolcor}
Let $M$ be a hypersurface of $\mathbb{Q}^3_\epsilon\times\mathbb{R}$ satisfying $ST=k_1 T$. 
Then, $(M,g)$ admits an almost Ricci soliton $(M,g,T,\lambda)$ if and only if the following equations are satisfied:
\begin{align}
\label{almostsoleq1}
    -\sin \theta e_1(\theta) - \cot{\theta} e_1(\theta) \left(\omega_{12}(e_2) + \omega_{13}(e_3) \right) + 2 \epsilon \sin^2 \theta  &= \lambda, \\
\label{almostsoleq2}
    \left(\cos \theta - \cot \theta e_1(\theta)\right)\omega_{12}(e_2) + \cot^2 \theta \omega_{12}(e_2)\omega_{13}(e_3) - \epsilon \cos^2 \theta &= \lambda, \\
\label{almostsoleq3}    
    \left(\cos \theta - \cot \theta e_1(\theta)\right)\omega_{13}(e_3) + \cot^2 \theta \omega_{12}(e_2)\omega_{13}(e_3) - \epsilon \cos^2 \theta &= \lambda.
\end{align}
\end{corollary}

\begin{lemma}
Let $M$ be a hypersurface of $\mathbb{Q}^3_\epsilon\times\mathbb{R}$ with the principal curvatures $k_1, k_2$ and $k_3$
satisfying $ST=k_1 T$. 
If $(M,g)$ admits an almost Ricci soliton $(M,g,T,\lambda)$, then we have the followings:
\begin{itemize}
\item [(i.)] $k_2=k_3$ and there is a smooth function $f$ such that $k_2=f(k_1)$,

\item [(ii.)] $k_2\neq k_3$ and there are smooth functions $f_1$ and $f_2$ such that $k_2=f_1(k_1)$ and $k_3=f_2(k_1)$. 
\end{itemize}
\end{lemma}

\begin{proof}
Assume that $M$ is a hypersurface of $\mathbb{Q}^3_\epsilon\times\mathbb{R}$ with the principal curvatures $k_1, k_2$ and $k_3$,
such that $ST=k_1 T$ and that $(M,g,T,\lambda)$ is an almost Ricci soliton. Then, equations \eqref{almostsoleq1}-\eqref{almostsoleq3} hold.
From \eqref{almostsoleq2} and \eqref{almostsoleq3}, we obtain 
\begin{equation}
(\cos{\theta}-\cot{\theta} e_1(\theta))(\omega_{12}(e_2)-\omega_{13}(e_3))=0.
\end{equation}
Thus, we have the following cases:
\\
\textit{Case(i.)} $\omega_{12}(e_2)=\omega_{13}(e_3)$, that is, $k_2=k_3$. Using Codazzi equations in \eqref{CodazziclassA}, 
we find $e_2(k_2)=e_2(k_3)=0$. Hence, the desired result follows.
\\
\textit{Case(ii.)} $e_1(\theta)=\sin{\theta}$. Considering $\omega_{12}(e_2)=k_2\tan{\theta}$ and $\omega_{13}(e_3)=k_3\tan{\theta}$,
we derive 
\begin{align}
\label{almostsoleq1_n}
-\sin^2{\theta}-\sin{\theta}(k_2+k_3)+2\epsilon\sin^2{\theta}&=\lambda,\\
\label{almostsoleq23_n}
k_2k_3-\epsilon\cos^2{\theta}&=\lambda.
\end{align}
From \eqref{almostsoleq1_n} and \eqref{almostsoleq23_n}, we get
\begin{equation}
\label{lasteq}
(\epsilon-1)\sin^2{\theta}+\epsilon-\sin{\theta}\;(k_2+k_3)-k_2k_3=0.
\end{equation}
Differentiating \eqref{lasteq} with respect to $e_1$, we have 
\begin{align}
\label{diflasteq}
&(4\varepsilon-2)\sin^2{\theta}\cos{\theta}+(2\epsilon-1)\cos{\theta}\sin{\theta} (k_2+k_3) 
+\tan{\theta}\sin{\theta}(k_2^2+k_3^2-k_1(k_2+k_3))\notag\\
&+k_2k_3\tan{\theta}(k_2+k_3-2k_1)=0.
\end{align}
Considering \eqref{lasteq} and \eqref{diflasteq} with $k_1=-e_1(\theta)$, it can be seen that $k_2$ and $k_3$ are functions of $\theta$. 
Since $e_2(\theta)=e_3(\theta)=0$, we have $e_2(k_2)=e_2(k_3)=e_3(k_2)=e_3(k_3)=0$. Thus, we get the obtained result. 
\end{proof}

In case (i), the classification is given by Theorem \ref{thmk2eqk3} in \cite{Dillen09}. 
Based on this result, we proceed to examine whether such rotational hypersurfaces of $\mathbb{Q}^3_\epsilon\times\mathbb{R}$ admit an almost Ricci soliton.

\begin{proposition}
\label{almostRiccirot}
Let $M$ be a rotational hypersurface of $\mathbb{Q}^3_\epsilon\times\mathbb{R}$.
Then, the followings hold:
\begin{itemize}
\item [i.] $(M,g, T, \lambda)$ is an almost Ricci soliton,
where the hypersurface $M$ of $\mathbb{S}^3\times\mathbb{R}$ is parametrized by
\begin{equation}
\label{rothypS3xR}
{\bf x}(s,v,w)= (\cos{s}, \sin{s}\cos{v}\sin{w}, \sin{s}\cos{v}\cos{w}, \sin{s}\sin{v}, a(s))
\end{equation}
for a smooth function $a(s)$ satisfying 
\begin{equation}
\label{rothypS3xRafunc}
a''(s)=\frac{(1+a'(s)^2)(a'(s)^2\cot^2{s}+a'(s)\cot{s}-a'(s)^2-2)}{1+a'(s)\cot{s}}
\end{equation}
and $\lambda$ is given by 
\begin{equation}
\label{rothypS3xRlambda}
\lambda = \frac{a'(s)(2a'(s)^2\cot^3{s}+3a'(s)\cot^2{s}-(2a'(s)^2+1)\cot{s}-a'(s))}{(1+a'(s)\cot{s})(1+a'(s)^2)},
\end{equation}

\item [ii.] $(M,g, T, \lambda)$ is an almost Ricci soliton
where the hypersurface $M$ of $\mathbb{H}^3\times\mathbb{R}$ is parametrized by
\begin{equation}
\label{rothypH3xR1}
{\bf x}(s,v,w) = (\cosh{s}\cosh{v},\cosh{s}\sinh{v}\sin{w},\cosh{s}\sinh{v}\cos{w},\sinh{s}, a(s))
\end{equation}
for a smooth function $a(s)$ satisfying 
\begin{equation}
\label{rothypH3xRa1func}
a''(s)=\frac{(1+a'(s)^2)(a'(s)^2\tanh^2{s}+a'(s)\tanh{s}-a'(s)^2-2)}{1+a'(s)\tanh{s}}
\end{equation}
and $\lambda$ is given by 
\begin{equation}
 \label{rothypH3xR1lambda}   
 \lambda = \frac{a'(s)(2a'(s)^2\tanh^3{s}+3a'(s)\tanh^2{s}-(2a'(s)^2+1)\tanh{s}-a'(s))}{(1+a'(s)\tanh{s})(1+a'(s)^2)},
\end{equation}

\item [iii.] $(M,g, T, \lambda)$ is an almost Ricci soliton
where the hypersurface $M$ of $\mathbb{H}^3\times\mathbb{R}$ is parametrized by
\begin{equation}
\label{rothypH3xR2}
{\bf x}(s,v,w) = (\cosh{s},\sinh{s}\cos{v}\sin{w},\sinh{s}\cos{v}\cos{w},\sinh{s}\sin{v},a(s))
\end{equation}
for a smooth function $a(s)$ satisfying 
\begin{equation}
\label{rothypH3xRa2func}
a''(s)=\frac{(1+a'(s)^2)(a'(s)^2\coth^2{s}+a'(s)\coth{s}-a'(s)^2-2)}{1+a'(s)\coth{s}}
\end{equation}
and $\lambda$ is given by
\begin{equation}
 \label{rothypH3xR2lambda}   
\lambda = \frac{a'(s)(2a'(s)^2\coth^3{s}+3a'(s)\coth^2{s}-(2a'(s)^2+1)\coth{s}-a'(s))}{(1+a'(s)\coth{s})(1+a'(s)^2)},
\end{equation}

\item [iv.] $(M,g, T, \lambda)$ is an almost Ricci soliton
where the hypersurface $M$ of $\mathbb{H}^3\times\mathbb{R}$ is parametrized by
\begin{equation}
\label{rothypH3xR3}
{\bf x}(s,v,w) = \left(s,sv,sw,-\frac{1}{2s}-\frac{s}{2}(v^2+w^2),a(s)\right)
\end{equation}
for a smooth function $a(s)$ satisfying 
\begin{equation}
\label{rothypH3xRa3func}
a''(s) =\frac{s^3a'(s)^3-3s^2a'(s)^2-2}{s^2(1+s\;a'(s))}
\end{equation}
and $\lambda$ is given by
\begin{equation}
 \label{rothypH3xR3lambda}   
 \lambda = \frac{sa'(s)(2s^3a'(s)^3-s^2a'(s)^2+2sa'(s)-1)}{(1+s^2a'(s)^2)^2}.
\end{equation}
\end{itemize}
\end{proposition}

\begin{proof}
Assume that $(M,g, T, \lambda)$ is an almost Ricci soliton, where 
the hypersurface $M$ is a rotational hypersurface in $\mathbb{S}^3\times\mathbb{R}$ given by \eqref{rothypS3xR}. 
From \cite{Dillen09}, we know that the principal curvatures $k_1$ and $k_2=k_3$ of $M$ are given by 
\begin{equation}
k_1=-\frac{a''(s)}{(1+a'(s)^2)^{3/2}},\;\;\;k_2=k_3=-\frac{a'(s)}{(1+a'(s)^2)^{1/2}}
\end{equation}
Then, equations \eqref{almostsoleq1} and \eqref{almostsoleq2} lead to \eqref{rothypS3xRlambda} and \eqref{rothypS3xRafunc},
which characterize the functions $a(s)$ and $\lambda$, respectively.
We provide the proof only for the rotational hypersurface in case (i), since the other cases can be treated in a similar way.
\end{proof}

For case (ii), Theorems \ref{classS3xRthreedistinct} and \ref{classH3xRthreedistinct} provide the complete classification. 
Based on these results, we examine whether the corresponding hypersurfaces admit an almost Ricci soliton.

\begin{proposition}
\label{almostRicciother}
There does not exist an almost Ricci soliton on $(M,g, T, \lambda)$, where $M$ is a hypersurface of $\mathbb{Q}^3_\epsilon\times\mathbb{R}$ 
given in Example \ref{exS3xR} and Example \ref{exH3xR}, respectively.    
\end{proposition}

\begin{proof}
\textit{Case(i.)}
Assume that $M$ is a hypersurface of $\mathbb{S}^3\times\mathbb{R}$ given by \eqref{positionexs3xR},
that is, $\epsilon=1$
and $(M,g,T,\lambda)$ is an almost Ricci soliton. 
From Example \eqref{exS3xR}, we know principal curvatures and connection forms of $M$. 
If we substitute these into \eqref{almostsoleq1} - \eqref{almostsoleq3}, we have the followings:
\begin{align} 
\label{s_inv_e1}
    \alpha_2''(s)+ \alpha^{\prime\prime}_1(s)(\tan(\alpha_1(s)) -\cot(\alpha_1(s))) +2\alpha'^2_1(s) = \lambda\\
\label{s_inv_e2}    
    (\alpha''_1(s) -\alpha'_1(s)\alpha'_2(s))\tan(\alpha_1(s))  - 2\alpha'^2_2(s)= \lambda\\
\label{s_inv_e3}
    (\alpha'_1(s)\alpha'_2(s) - \alpha''_1(s))\cot(\alpha_1(s)) - 2\alpha'^2_2(s)= \lambda.
\end{align}
From \eqref{s_inv_e2} and \eqref{s_inv_e3}, we find 
\begin{equation}
\alpha_1''(s)=\alpha_1'(s)\alpha_2'(s),\;\; \lambda=-2\alpha'^2_2(s).
\end{equation}
Substituting $\lambda$ into the equation \eqref{s_inv_e1}, we obtain
\begin{align} 
\label{s_inv_e1_fullsimplify}
    2 + \alpha''_2(s) = \alpha''_1(s)(\cot{(\alpha_1(s))}-\tan{(\alpha_1(s))}).
\end{align}
If $(M,g,T,\lambda)$ is a Ricci soliton, then $\lambda=-2(\alpha'_2(s))^2$ is a constant. 
On the other hand, $\alpha_1'(s)^2+\alpha_2'(s)^2=1$ implies that
$\alpha_1'(s)$ is also a constant. From the equation \eqref{s_inv_e1_fullsimplify},
we get a contradiction. Thus, $\lambda$ can not be a constant.
Now, we will show that $(M,g,T,\lambda)$ is not an almost Ricci soliton. 
Considering $\alpha_1''(s)=\alpha_1'(s)\alpha_2'(s)$ and $\alpha_1'(s)\alpha_1''(s)+\alpha_2'(s)\alpha_2''(s)=0$,
the equation \eqref{s_inv_e1_fullsimplify} becomes 
\begin{equation}
\label{propeq1}
\cot{(2\alpha_1(s))}=\frac{2-\alpha_1'(s)^2}{2\alpha_1'(s)\alpha_2'(s)}.
\end{equation}
Moreover, the equation \eqref{s_inv_e1} gives 
\begin{equation}
\label{propeq2}
\cot^2(2\alpha_1(s)) =\frac{-\lambda^2+4\lambda-4}{4\lambda^2+8\lambda}.
\end{equation}
Differentiating \eqref{propeq1} with respect to $s$, we obtain 
\begin{equation}
\label{propeq3}
\cot^2{(2\alpha_1(s))}=\frac{-4\lambda^2-2\lambda+4}{4\lambda^2+8\lambda}.
\end{equation}
From \eqref{propeq2} and \eqref{propeq3}, we get that $\lambda$ satisfies $3\lambda^2+6\lambda-8=0$
which means $\lambda$ is a real constant. On the other hand, we know that $\lambda$ can not be a constant. 
Thus, $(M,g,T,\lambda)$ is not an almost Ricci soliton.

\textit{Case(ii.)}
Assume that $M$ is a hypersurface of $\mathbb{H}^3\times\mathbb{R}$ given by \eqref{positionexH3xR}
and $(M,g,T,\lambda)$ is an almost Ricci soliton. 
From Example \eqref{exH3xR}, we know principal curvatures and connection forms of $M$. 
If we substitute these into \eqref{almostsoleq1} - \eqref{almostsoleq3}, we have the followings:
\begin{align} 
\label{h_inv_e1}
   \alpha_2''(s) - \alpha''_1(s)\left(\tanh(\alpha_1(s)) + \coth(\alpha_1(s))\right) - 2\alpha'_1(s)^2= \lambda,\\
\label{h_inv_e2}    
    (\alpha'_1(s)\alpha'_2(s)- \alpha''_1(s))\tanh(\alpha_1(s))  + 2 \alpha'_2(s)^2 = \lambda\\
\label{h_inv_e3}
    (\alpha'_1(s)\alpha'_2(s) - \alpha''_1(s))\coth(\alpha_1(s))  + 2 \alpha'_2(s)^2 = \lambda.
\end{align}
From \eqref{h_inv_e2} and \eqref{h_inv_e3}, we find 
\begin{equation}
\alpha_1''(s)=\alpha_1'(s)\alpha_2'(s),\;\; \lambda=2 \alpha'_2(s)^2.
\end{equation}
Substituting $\lambda$ into the equation \eqref{h_inv_e1}, we obtain
\begin{align} \label{h_inv_e1_fullsimplify}
   \alpha''_2(s) -2 = \alpha''_1(s)\left(\tanh(\alpha_1(s))+\coth(\alpha_1(s))\right).
\end{align}
Similarly, if $(M,g,T,\lambda)$ is a Ricci soliton, the equation \eqref{h_inv_e1_fullsimplify} gives a contradiction.
Thus, it can not be a Ricci soliton. 
Now, we will show that $(M,g,T,\lambda)$ is not an almost Ricci soliton. 
Considering $\alpha_1''(s)=\alpha_1'(s)\alpha_2'(s)$ and $\alpha_1'(s)\alpha_1''(s)+\alpha_2'(s)\alpha_2''(s)=0$,
the equation \eqref{h_inv_e1_fullsimplify} becomes 
\begin{align}
\label{propeq1h}
   \coth{(2\alpha_1(s))}=\frac{-2-\alpha_1'(s)^2}{2\alpha_1'(s)\alpha_2'(s)}.
\end{align}
Moreover, the equation \eqref{h_inv_e1} gives 
\begin{equation}
\label{propeq2_h}
\coth^2(2\alpha_1(s)) = \frac{(\lambda-6)^2}{4\lambda(2-\lambda)}.
\end{equation}
Differentiating \eqref{propeq1h} with respect to $s$, we obtain 
\begin{align*}
    \coth^2(2\alpha_1(s)) =\frac{2\lambda^2+\lambda-6}{2\lambda(\lambda-2)}.
\end{align*}
Considering these equations, we obtain 
\begin{align} \label{cond_on_lambda} 
   5\lambda^2-10\lambda+24=0
\end{align}
which gives a contradiction.  
Thus, $(M,g,T,\lambda)$ is not an almost Ricci soliton.
\end{proof}

Using Proposition \ref{almostRiccirot} and Proposition \ref{almostRicciother}, we give the following theorem, directly.

\begin{theorem}
Let $M$ be a hypersurface of $\mathbb{Q}^3_\epsilon\times\mathbb{R}$ satisfying $ST=k_1 T$. 
Then,  $(M,g, T, \lambda)$ is an almost Ricci soliton if and only if $M$ is a rotational hypersurface of $\mathbb{Q}^3_\epsilon\times\mathbb{R}$
described in Proposition \ref{almostRiccirot}. Moreover,  $(M,g, T, \lambda)$ is a gradient almost Ricci soliton $(M,g, h, \lambda)$, where $h$ is the height function  defined by \eqref{Sect4Rem1Eq1}.
\end{theorem}

\section*{Acknowledgements}
This work was carried out during the 1001 project supported by the Scientific and Technological Research Council of T\"urkiye (T\"UB\.ITAK)  (Project Number: 121F352). This article is based on research conducted as part of the first author's PhD dissertation.

\section*{Declarations}


\textbf{Data Availability.} Data sharing not applicable to this article because no datasets were generated or analysed during the current study.

\textbf{Funding.} The authors have not disclosed any funding.

\textbf{Code availability.} N/A.

\textbf{Conflicts of interest.} The authors have not disclosed any competing interests.

\end{document}